\documentclass[11pt]{amsart}

\usepackage[ansinew]{inputenc}
\usepackage[a4paper,dvips]{geometry}
\usepackage{latexsym}
\usepackage{amsfonts}
\usepackage{amsmath,amssymb}
\usepackage{bbm}
\usepackage[dvipsnames]{xcolor}
\usepackage{tikz}
\usetikzlibrary{decorations.pathreplacing,angles,quotes}

\usepackage{fixme}
\usepackage{graphicx}
\allowdisplaybreaks

\makeatletter
\newcommand{\addresseshere}{%
  \enddoc@text\let\enddoc@text\relax
}
\makeatother

\newtheorem{theorem}{Theorem}
\newtheorem{remark}[theorem]{Remark}
\newtheorem{lem}[theorem]{Lemma}
\newtheorem{prop}[theorem]{Proposition}

\newcommand{\bv}[1]{\mathbf{#1}}
\DeclareMathOperator{\M}{M}
\DeclareMathOperator{\conv}{conv}

\newcommand{\NN}{\mathbb{N}}

\newcommand{\FF}{\mathbb{F}}

\newcommand{\Var}{{\mathbb V}\mathrm{ar}}
\newcommand{\E}{{\mathbb E}}
\newcommand{\cP}{{\mathcal P}}

\newcommand{\cL}{{\mathcal L}}

\newcommand{\N}{{\mathbb N}}
\newcommand{\R}{{\mathbb R}}

\newcommand{\bq}{{\mathbf{q}}}
\newcommand{\bOmega}{{\mathbf{\Omega}}}
\newcommand{\bx}{{\mathbf{x}}}

\newcommand{\by}{{\mathbf{y}}}
\newcommand{\bz}{{\mathbf{z}}}
\newcommand{\1}{\mathrm{1}}
\newcommand{\dd}{\mathrm{d}}
\newcommand{\ba}{\mathbf{a}}

\newcommand{\bo}{\mathbf{0}}

\DeclareFontFamily{U}{skulls}{}
\DeclareFontShape{U}{skulls}{m}{n}{ <-> skull }{}

\newcommand{\revised}[1]{\textcolor{black}{#1}}

\title[Jittered sampling is not optimal]{\large On a partition with a lower expected $\cL_2$-discrepancy than classical jittered sampling}

\author{Markus Kiderlen}
\address{Aarhus University, Aarhus, Denmark}
\email{kiderlen@math.au.dk}

\author{Florian Pausinger}
\address{Queen's University Belfast, Belfast, United Kingdom.}
\email{f.pausinger@qub.ac.uk}

\date{}

\begin{document}


%%%%%%%%%%%%%%%%%%%
%	Abstract
%%%%%%%%%%%%%%%%%%%

\begin{abstract}
	We prove that classical jittered sampling of the $d$-dimensional unit cube 
	does \emph{not} yield the smallest expected $\cL_2$-discrepancy among all stratified samples with $N=m^d$ points. Our counterexample can be given explicitly and consists of convex partitioning sets of equal volume.
\end{abstract}

%%%%%%%%%%%%%%%%%%%%
% Maketitle
%%%%%%%%%%%%%%%%%%%%
\maketitle

%%%%%%%%%%%%%%%%%%%
%	Introduction
%%%%%%%%%%%%%%%%%%%

\section{Introduction}

%\begin{enumerate}
%\item \fp{Explain stratified sampling}
%\item \fp{Formulate our main result --  $\cL_p$ result, formulate and prove special $\cL_2$ example in section 3, moment thm,}
%\item \fp{Explain proof idea}
%\item \fp{Outline}
%\end{enumerate}

Classical \emph{jittered sampling}  with $N=m^d$ points combines the simplicity of grids with uniform random sampling by partitioning $[0,1]^d$ into $m^d$ axis-aligned congruent cubes and placing a random point inside each of them; see Fig.~\ref{fig:def} (left). Jittered sampling is sometimes referred to as `stratified sampling' in the literature, but we will use the term `stratified sampling' in a more broad sense.  
\revised{Let $\bOmega=(\Omega_1,\ldots,\Omega_N)$ denote a general partition of $[0,1]^d$ 
	into $N$ subsets $\Omega_1,\ldots,\Omega_N$ of positive volume. A \emph{stratified sample}  $\cP=\cP_{\bOmega}$, based on this partition, is a set of $N$ random points,  where 
	the $i$th point in $\cP$ is chosen uniformly in the $i$th set of the partition (and stochastically independent of the other points), $i=1,\ldots,N$. 
} 
If $N=m^d$ and the partition consists of the above mentioned axis-aligned congruent cubes, we obtain \emph{jittered sampling} as a special case.

 To analyse the irregularities of such points the concept of $\cL_p$-discrepancy is commonly used.
Given \revised{a set $\cP \subset [0,1]^d$ of $N$ points}
 and a vector $\bv x = (x_1, \ldots, x_d)$ in the unit cube, the discrepancy function
	\begin{align}\label{eq:1} 
	d_\cP(\bv x)=\frac{\#\left(\cP\cap[0, {\bv x}]\right)}{N} - \big|[0, \bv x]\big|
	\end{align}
\revised{evaluates to the difference between the relative number of points in an axis-aligned box of the form $[0,\bv x] :=\prod_{k=1}^d [0,x_k]$ and the volume of this box.}
%is the deviation of the empirical distribution of the point set from the uniform distribution 
Here, $\big|\cdot\big|$ denotes the  Lebesgue measure and $\#\left(\cP\cap[0, {\bv x}]\right)$ counts the number of points \revised{of $\cP$ that lie in $[0, {\bv x}]$.}
For $1\le p<\infty$, the  $\cL_p$-discrepancy 
\[ \cL_{p} (\cP) := \|d_\cP\|_p
\]
of $\cP$ is defined as the $L_p$-norm $\|\cdot\|_p$ of the discrepancy function. We will only work with $p=2$. 

%For an infinite sequence $\cS$ of points in $[0,1]^d$ the $\cL_p$-discrepancy $\cL_{p}(\cS_N)$ is the $\cL_p$-discrepancy of the first $N$ elements, $\cS_N$, of $\cS$.
As a side remark, the well-known \emph{star discrepancy} %\mk{$\cL_{\infty} (\cP)$ of $\cP$} 
can be defined as the $L_{\infty}$-norm of the discrepancy function and is \revised{generally much} harder to study.

\subsection{Optimal $\cL_p$-discrepancy bounds.} The $\cL_p$-discrepancy, and in particular the special case $p=2$, is a well studied and understood measure for the irregularities of point sets. We refer to the book \cite{DP} and \revised{the survey \cite{DP2}} for further details. In particular, and in contrast to other measures such as the star-discrepancy, it is known how to construct deterministic point sets with the optimal order of magnitude of the $\cL_p$-discrepancy; see \cite{chen1,DP2,DP3}. 
In fact, for $d\geq2$ 
%the optimal order of the $\cL_p$-discrepancy for finite point sets is known and 
there are constructions of point sets
\revised{$\cP \subset [0,1]^d$} of $N$ points such that 
$$\cL_{p} (\cP) = \Theta \left(\frac{(\log N)^{\frac{d-1}{2}}}{N}\right).$$ 
The optimality of these constructions follows from a seminal result of Roth \cite{roth} who derived a general lower bound for the $\cL_2$-discrepancy of arbitrary sets of $N$ points in $[0,1]^2$ which was later generalised to $d\geq 2$ and $1 < p < \infty$; see \cite[Section 3.2]{DP}.
While deterministic point sets with small discrepancy are widely used in the context of numerical integration, simulations of different real world phenomena may require an element of randomness. The expected (star) discrepancy of a set $\cP_N$ of $N$ i.i.d.~uniform random points in $[0,1]^d$ is of order $\Theta(\sqrt{1/N})$; \revised{see \cite{hnww} for the first upper bound, \cite{aist} for the first upper bound with explicit constant and \cite{doerr1} for the first lower bound as well as \cite{gnewuch} for the current state of the art results in this context.} 
\revised{This can be compared to 
	a recent result by Doerr \cite{doerr} on the precise asymptotic order of the expected star-discrepancy of a point set obtained from jittered sampling:
	$$ \E \cL_{\infty} (\cP_N) = \Theta \left( \frac{\sqrt{d} \sqrt{1+\log(N/d)}}{N^{\frac{1}{2}+\frac{1}{2d} }} \right).$$ 
	That jittered sampling is (asymptpotically) better than equally many i.i.d.~uniform random points is not only true in terms of the star-discrepancy. 
	In fact, for $1<p<\infty$,} a stratified set derived from a partition into $N>2$ equivolume sets \emph{always} has a smaller expected $\cL_p$-discrepancy than a set consisting of $N$ i.i.d.~random points. 
	\revised{ This \emph{strong partition principle} was proven in \cite[Theorem 1]{mf21} (see also \cite{stefan1} for a weaker form) and raises the question which partition yields the stratified sample with the smallest mean $\cL_p$-discrepancy -- if such a partition exists. 
For simplicity we will only consider $p=2$ here.} 
\subsection{Generalized $\cL_p$-discrepancy}
\revised{Of course, any such analysis should keep the observation of Matou\v{s}ek \cite{mat} in mind that the $\cL_2$-discrepancy can be misleading if the dimension $d$ is high and the number of points is relatively small, i.e.~smaller than $2^d$. 
He observed that point sets with few points that cluster close to $(1, \ldots, 1)$ can have nearly optimal $\cL_2$-discrepancy \cite[Section 2]{mat}.
This, and the fact that the $\cL_p$-discrepancy does not yield a satisfying version of the classical Koksma-Hlawka Theorem, motivated Hickernell \cite{hick2} to extend the notion. 
Hickernell's generalization incorporates not only the original $\cL_p$-discrepancy but also the discrepancies of all projections of the point set to lower dimensional faces of the unit cube. We will briefly return to this more general discrepancy in Section \ref{sec:conclusion}.}
%

%\subsection{Other discrepancies}\fp{Florian moved that discussion to conclusion section. but can also be here.}

\subsection{Stratified sampling.}
\revised{Before detailing our result, we want to emphasize that our notion of \emph{stratified sampling} is a special case of the stratification tool usually employed in statistics and simulation; see, for instance \cite[Section 4.3.4]{rubins} (and the corresponding notion 
	for finite populations in	 e.g.~\cite{thompson}). There, 
the underlying cube $[0,1]^d$ is partitioned into $k\le N$ sets $\Omega_1,\ldots,\Omega_k$, 
and $N_i\ge1$ points are sampled i.i.d.~uniformly in $\Omega_i$ independently of the other points. 
The resulting stratified sample $\tilde \cP$ consisting of $N=N_1+\cdots+N_k$ points is only 
covered by our definition when $k=N$, in which case exactly one point is allocated to 
each of the strata, i.e.~$N_1=\cdots=N_k=1$. Often, an allocation proportional to stratum size is used, that is, one requires  that $N_i=N |\Omega_i|$ for all $i=1,\ldots,k$. 
For given $N$, the question if there is a $k\le N$, and a partition  $\Omega_1,\ldots,\Omega_k$  
such that the sample $\tilde \cP$ with 
 allocation proportional to stratum size minimizes the expected $\cL_2$-discrepancy $\E {\cL}_2^2(\tilde \cP)$, is only seemingly more general than 
asking for a minimizing equivolume partition with $N$ strata, where only one point is sampled in each. This is a consequence of a version of the strong partition principle -- see Remark \ref{rem}, below -- and is the reason why we restrict attention to one sampling point per stratum. The existence of such minimizing equivolume partitions is not trivial, but has been shown in \cite{mf21} under additional regularity assumptions on the sets $\Omega_i$. For instance, a minimizer exists among all \emph{convex} equivolume partitions; see \cite[Corollary 1]{mf21}. 
}
\subsection{Our result.}
\revised{When $N=m^d$ for some $m\in \N$,} it is natural to ask whether classical jittered sampling yields such a minimiser among all equivolume stratified samples. Our main result gives a \emph{negative} answer to this question, \revised{even if we require that all strata are convex.} 

 %To prove that 
 %classical jittered sampling does not yield the smallest expected $\cL_2$-discrepancy among all convex equivolume partitions with the same number of sets, 
In particular, we construct a convex, equivolume partition for each $N$ and show that this partition improves the expected $\cL_2$-discrepancy of jittered sampling with the same number of points.  It should be noted that we do not aim to maximise this improvement but rather to give an elementary and intuitive exposition. 
%\mk{The new partition, better than the partition for jittered sampling, can itself be improved as we will outline in Section \ref{sec:exploit}.}
\revised{This new partition can itself be improved as we will outline in Section \ref{sec:exploit}.}

\begin{center}
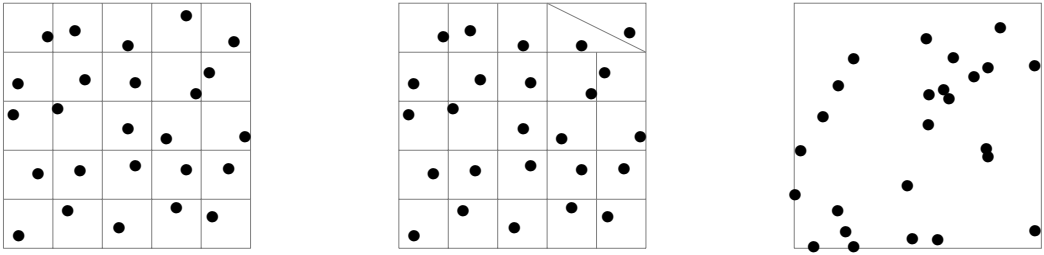
\begin{figure}[h!]
\centering
\begin{tikzpicture}[scale=0.65]
\draw[step=1cm,gray,very thin] (0,0) grid (5,5);
\draw[step=1cm,gray,very thin] (8,0) grid (13,4);
\draw[gray, very thin] (8,4) -- (8,5) -- (13,5) -- (13,4);
\draw[gray, very thin] (9,5) -- (9,4);
\draw[gray, very thin] (10,5) -- (10,4);
\draw[gray, very thin] (11,5) -- (11,4);
\draw[gray, very thin] (11,5) -- (13,4);
\draw[gray,very thin] (16,0) -- (21,0) -- (21,5) -- (16,5)--(16,0);
%left point set
\node at (0.31,0.23) {$\bullet$}; 
\node at (0.7,1.5) {$\bullet$}; 
\node at (0.2,2.7) {$\bullet$}; 
\node at (0.3,3.34) {$\bullet$}; 
\node at (0.9,4.29) {$\bullet$}; 

\node at (1.3,0.74) {$\bullet$}; 
\node at (1.55,1.56) {$\bullet$}; 
\node at (1.1,2.83) {$\bullet$}; 
\node at (1.65,3.42) {$\bullet$}; 
\node at (1.45,4.43) {$\bullet$}; 

\node at (2.34,0.41) {$\bullet$}; 
\node at (2.67,1.66) {$\bullet$}; 
\node at (2.52,2.42) {$\bullet$}; 
\node at (2.67,3.37) {$\bullet$}; 
\node at (2.52,4.12) {$\bullet$}; 

\node at (3.5,0.82) {$\bullet$}; 
\node at (3.7,1.58) {$\bullet$}; 
\node at (3.3,2.21) {$\bullet$}; 
\node at (3.9,3.14) {$\bullet$}; 
\node at (3.7,4.72) {$\bullet$}; 

\node at (4.23,0.63) {$\bullet$}; 
\node at (4.56,1.61) {$\bullet$}; 
\node at (4.89,2.25) {$\bullet$}; 
\node at (4.17,3.57) {$\bullet$}; 
\node at (4.67,4.19) {$\bullet$}; 

%middle point set
\node at (8.31,0.23) {$\bullet$}; 
\node at (8.7,1.5) {$\bullet$}; 
\node at (8.2,2.7) {$\bullet$}; 
\node at (8.3,3.34) {$\bullet$}; 
\node at (8.9,4.29) {$\bullet$}; 

\node at (9.3,0.74) {$\bullet$}; 
\node at (9.55,1.56) {$\bullet$}; 
\node at (9.1,2.83) {$\bullet$}; 
\node at (9.65,3.42) {$\bullet$}; 
\node at (9.45,4.43) {$\bullet$}; 

\node at (10.34,0.41) {$\bullet$}; 
\node at (10.67,1.66) {$\bullet$}; 
\node at (10.52,2.42) {$\bullet$}; 
\node at (10.67,3.37) {$\bullet$}; 
\node at (10.52,4.12) {$\bullet$}; 

\node at (11.5,0.82) {$\bullet$}; 
\node at (11.7,1.58) {$\bullet$}; 
\node at (11.3,2.21) {$\bullet$}; 
\node at (11.9,3.14) {$\bullet$}; 
\node at (11.7,4.12) {$\bullet$}; 

\node at (12.23,0.63) {$\bullet$}; 
\node at (12.56,1.61) {$\bullet$}; 
\node at (12.89,2.25) {$\bullet$}; 
\node at (12.17,3.57) {$\bullet$}; 
\node at (12.67,4.39) {$\bullet$}; 

%right point set
\node at (17.2039, 0.0104) {$\bullet$}; 
\node at (19.2219, 3.8681) {$\bullet$}; 
\node at (19.1364, 3.0262) {$\bullet$}; 
\node at (19.924, 1.8585) {$\bullet$}; 
\node at (16.5854, 2.6663) {$\bullet$}; 

\node at (18.7159, 2.4971) {$\bullet$}; 
\node at (19.8897, 2.0048) {$\bullet$}; 
\node at (20.8751, 0.3469) {$\bullet$}; 
\node at (18.674, 4.25) {$\bullet$}; 
\node at (20.1753, 4.479) {$\bullet$}; 

\node at (16.8948, 3.299) {$\bullet$}; 
\node at (19.6393, 3.4874) {$\bullet$}; 
\node at (18.7296, 3.1135) {$\bullet$}; 
\node at (17.2053, 3.8576) {$\bullet$}; 
\node at (20.8653, 3.7005) {$\bullet$}; 

\node at (18.2948, 1.2544) {$\bullet$}; 
\node at (18.9047, 0.153) {$\bullet$}; 
\node at (16.0157, 1.0807) {$\bullet$}; 
\node at (18.3935, 0.1798) {$\bullet$}; 
\node at (19.9259, 3.6681) {$\bullet$}; 

\node at (16.8825, 0.7463) {$\bullet$}; 
\node at (16.3982, 0.0242) {$\bullet$}; 
\node at (16.1357, 1.9686) {$\bullet$}; 
\node at (19.0286, 3.2103) {$\bullet$}; 
\node at (17.0463, 0.3119) {$\bullet$}; 
\end{tikzpicture}
\caption{\emph{Left:} Jittered sampling for $d=2$ and $m=5$. \emph{Middle:}  Improvement of jittered sampling. \emph{Right:} $N=25$ i.i.d. uniform random points.} \label{fig:def}
\end{figure}
\end{center}

\begin{theorem} \label{thm:main}
Let $m \in \NN$, $m\geq 2$ be given. The set of $m^d$ random points in $[0,1]^d$ obtained from jittered sampling does not have the minimal expected $\cL_2$-discrepancy among all stratified samples from convex equivolume partitions with the same number of points.
\end{theorem} 

%\begin{theorem} \label{thm:L2}
%	Let $m\in \N$, $m\ge2$ be given. Then the set of  $m^d$ random points in $[0,1]^d$ obtained from jittered sampling does not have the minimal expected mean squared discrepancy among all stratified samples from convex equivolume partitions with the same number of points. 
%\end{theorem}

%Section \ref{sec:prelim} contains background material from \cite{mf21}.
\subsection{Outline} 
In the next section we modify the partition of jittered sampling to obtain another partition, depicted for $d=2$ in Figure \ref{fig:def} (middle), and  show that this new partition gives rise to a better stratified point set. \revised{We further explore this construction in Section \ref{sec:exploit} and discuss the maximal gain that can be achieved with our method for $d=2$.}
Various open questions and directions for future research are discussed in Section \ref{sec:conclusion}. 

%%%%%%%%%%%%%%%%%%%
%	Preliminaries
%%%%%%%%%%%%%%%%%%%

\section{An elementary proof of Theorem \ref{thm:main}} \label{sec:L2}

%\begin{enumerate}
%\item \fp{Recall important results and notations from our first paper.}
%\end{enumerate}

 To begin with, we summarize  a number of known results for stratified sampling. 
	As we will need them for  underlying sets other than the unit cube, we consider 
	in the following stratified samples in a fixed compact convex set $K\subset \R^d$ with $|K|>0$. 
Let  $\bOmega$ be a \emph{partition of $K$} into $N\in \N$ \revised{sets 
$\Omega_1,\ldots,\Omega_N$ of positive volume}. This means that $\Omega_1\cup\cdots\cup\Omega_N=K$ and that two different partitioning sets have no interior points in common. 
Let $\cP_\bOmega$ be the corresponding  stratified sample. 
We will call a partition \emph{convex} if all the sets $\Omega_1,\ldots,\Omega_N$ are convex.

	Generalizing \eqref{eq:1}
the discrepancy function $d_{\cP}(\bv x)$ of a finite set of points $\cP=\{\bx_1,\ldots,\bx_n\}\subset K$ 
 is given by 
	\begin{align}\label{eq:1prime} 
	d_\cP(\bv x)=\frac{\#\left(\cP\cap(-\infty, {\bv x}]\right)}{N} - 
	\frac{\big|K\cap (-\infty, \bv x]\big|}{|K|},
\end{align}
$\bv x\in K$ and $(-\infty,\bv x] :=\prod_{k=1}^d (-\infty,x_k]$. If $\cP_\bOmega$ is a stratified sample based on the partition $\bOmega$ of $K$, the discrepancy function at $\bx$ is a random variable. 
It was shown in \cite[Proposition 1]{mf21} for $K=[0,1]^d$ that $d_{\cP_\bOmega}(\bv x)$
has mean $0$ for all $\bx\in K$
if and only if the partition is \emph{equivolume}, that is, if $|\Omega_1|=\cdots=|\Omega_N|$. 
The proof extends literally to all compact convex $K$ with positive volume. 
From now on, all the partitions  $\bOmega$ we consider in this paper will be equivolume.

 For $p\ge 1$, the \emph{mean $\cL_p$-discrepancy} is usually defined as  $\E\cL_{p}^p (\cP)=\E \|d_\cP\|_p^p$, where $\|\cdot\|_{p}$ is now the $L_p$-norm of functions on $K$ with respect to the uniform distribution on $K$.  It should correctly be called `mean $p$-th power $\cL_p$-discrepancy'. %For $p=2$ an explicit integral expression can be obtained.
		Due to Tonelli's theorem we see that
\[
\E {\cL}_p^p(\cP_\bOmega)=\frac1{|K|}\int_{K} \E [d_{\cP_\bOmega}(\bv x)]^p \dd\bx. 
\]
 As  $\bOmega$ is equivolume, 	$d_{\cP_\bOmega}(\bv x)$ is centered, so  
\begin{equation}\label{eq:pthcnteredmean} 
	\E {\cL}_p^p(\cP_\bOmega)=\frac1{|K|}\int_{K} \M_p\left(\frac{\#\left(\cP_\bOmega\cap(-\infty, {\bv x}]\right)}{N}\right) \dd\bx,
\end{equation}
where 
\[
\M_p(Y)=\E\big|Y-\E Y\big|^p
\]
is the $p$th central moment of a random variable $Y$. The variable  $\#(\cP\cap(-\infty, {\bv x}])$ is the sum of $N$ independent (but  not identically distributed) Bernoulli variables with success probabilities $q_1(\bx),\ldots,$ $q_N(\bx)$, where 
\begin{equation}\label{eq:qi}
q_i(\bx)=\frac{|\Omega_i\cap (-\infty,\bx]|}{|\Omega_i|}=\frac N{|K|}|\Omega_i\cap (-\infty,\bx]|.
\end{equation}
The distribution of  $\#(\cP\cap(-\infty, {\bv x}])$ is usually called \emph{Poisson-binomial distribution} 
with $N$ trials and parameter vector $\bq(\bx)=(q_1(\bx),\ldots,q_N(\bx))$. Its mean is 
\begin{equation}\label{eq:mean}
\sum_{i=1}^N q_i(\bx)=\frac{N}{|K|}|K\cap (-\infty,\bx]|.
\end{equation}
\revised{
\begin{remark}
	\label{rem}
	Similar arguments also apply to the generalized stratified sample $\tilde \cP$ based on a partition 
	$\Omega_1,\ldots,\Omega_k$ of $[0,1]^d$ with allocation proportional to size ($N_i=N |\Omega_i|$ for $i=1,\ldots,k$), as outlined in the introduction. The equivalent of 
	\eqref{eq:pthcnteredmean}  is now 
\begin{equation*}%\label{eq:pthcnteredmean1} 
	\E {\cL}_p^p(\tilde \cP)=\int_{[0,1]^d} \M_p\left(\frac{\#\left(\tilde \cP\cap[0, {\bv x}]\right)}{N}\right) \dd\bx,
\end{equation*}
and specializes for $p=2$ to 
\begin{equation}\label{eq:pthcnteredmean1} 
	\E {\cL}_2^2(\tilde \cP)=\frac1{N^2}\sum_{i=1}^k\int_{[0,1]^d} \Var\big(
	{\#(\tilde \cP\cap\Omega_i\cap[0, {\bv x}])}\big) \dd\bx,
\end{equation}
where the independence of sampling points in different strata was used. Equation 
\eqref{eq:pthcnteredmean1}  shows that the mean $\cL_2$-discrepancy is a sum of integrated variances originating from the $k$ strata. In stratum $\Omega_i$ the $N_i$ points are i.i.d. uniform, and \cite[Proposition 4.3.1]{rubins} implies that 
sub-stratifying this set into $N_i$ equivolume subsets, and choosing one point 
uniformly in each of these sub-strata does not increase the variance. Applying this to 
all strata $\Omega_i$ with $N_i>1$ yields a new stratified sample based on $N$ equivolume 
strata with a mean $\cL_2$-discrepancy that is not larger than 
\eqref{eq:pthcnteredmean1}. 
\end{remark}
}

As stratified points are independent and $M_2$ is additive for independent variables, \revised{\eqref{eq:pthcnteredmean}} allows us to state the following explicit formula when $p=2$ which was stated in \cite[Proposition 2]{mf21} for $K=[0,1]^d$.

\begin{prop}\label{prop:3}
 	If $\bOmega$ is an equivolume partition of a compact convex set $K\subset \R^d$ with $|K|>0$ then 
	\begin{equation*}
\E {\cL}_2^2(\cP_\bOmega)=\frac1{N^2|K|}\sum_{i=1}^N \int_{K}  
q_i(\bx)\big(1-q_i(\bx)\big)
 \dd\bx,
\end{equation*}
where $q_i(\bx)$ is given in \eqref{eq:qi}.
\end{prop}
 The relation in Proposition \ref{prop:3} actually shows that only the integrals of 
	$q_i^2$ are required, as \eqref{eq:mean} gives 
\begin{equation}\label{eq_q2}
	\E {\cL}_2^2(\cP_\bOmega)=\frac1{N|K|^2} \int_{K}  |K\cap (-\infty,\bx]|
	\dd\bx-\frac1{N^2|K|}\sum_{i=1}^N \int_{K}  
	q_i^2(\bx)
	\dd\bx. 
\end{equation}	

%%%%%%%%%%%%%%%%%%%
%	Case L_2
%%%%%%%%%%%%%%%%%%%

%\begin{enumerate}
%\item \fp{Lemma: Case $d=2$}
%\item \fp{Lemma: General case $d\geq 2$}
%\item \fp{Prove version of main theorem for $L_2$ discrepancy}
%\item \fp{IMPORTANT: The elementary method has also the advantage that we get the actual values for the expected values. When using majorizations, we do not get these values but just the inequality. This needs to be highlighted and this justifies this section. Moreover, it justifies why we later reprove this case using majorizations in the form ot thm 2.4 of tang and tang.}
%\item \fp{related to previous point, maybe formulate two main results. One with the actual difference between the two partitions in the case of the $L_2$ discrepancy and a second general one in which we only require a majorization. As a corollary of second result we can derive inequality for example in first theorem. As a third main result, which should be stated as second theorem between the other two, we have that the $k$-th moments of a Poisson Binomial rv is Schur concave.}
%\end{enumerate}

\begin{center}
\begin{figure}[h!]
\centering
\begin{tikzpicture}[scale=0.65]
\draw[black, thin] (0,0) -- (6,0) -- (6,3) -- (0,3) -- (0,0);
\draw[black, thin] (3,0) -- (3,3);
\draw[black, thin] (9,0) -- (15,0) -- (15,3) -- (9,3) -- (9,0);
\draw[black, thin] (9,3) -- (15,0);

%\draw[black, very thin] (0,2)--(4.2,2) -- (4.2,0);
%\draw[black, very thin] (9,2)--(13.2,2) -- (13.2,0);

\node at (1.5, 1.5) {$\Omega_{1,|}$}; 
\node at (4.5, 1.5) {$\Omega_{2,|}$}; 
\node at (10.5, 1) {$\Omega_{1,\backslash}$}; 
\node at (13.5, 2) {$\Omega_{2,\backslash}$}; 

\node at (-0.5, -0.5) {\tiny $(a_1,a_2)$}; 
\node at (6.5, -0.5) {\tiny $(a_1+2b,a_2)$}; 
\node at (-0.5, 3.5) {\tiny $(a_1,a_2+b)$}; 
%\node at (3, -0.5) {\tiny $(a_1 + b,a_2)$}; 
%\node at (13.5, 2) {$\Omega_{2,\backslash}$}; 

\node at (0, 0) {$\bullet$}; 
\node at (6, 0) {$\bullet$}; 
\node at (0, 3) {$\bullet$}; 
%\node at (3, 0) {$\bullet$}; 

\node at (9, 0) {$\bullet$}; 
\node at (15, 0) {$\bullet$}; 
\node at (9, 3) {$\bullet$}; 
%\node at (3, 0) {$\bullet$}; 

\end{tikzpicture}
\caption{The  two partitions of the rectangle studied in Lemma \ref{lem1}. } \label{fig:rectangle}
\end{figure}
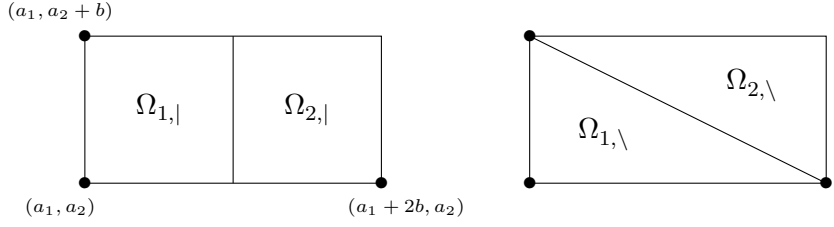
\end{center}

 The proof of Theorem \ref{thm:main} can be reduced to a  comparison of two partitions in the plane depicted in Figure \ref{fig:rectangle}. 
\begin{lem}\label{lem1}
	Let $\ba=(a_1,a_2)\in \R^2$ and $b>0$ be given. Consider the rectangle \[
	I=[a_1,a_1+2b]\times [a_2,a_2+b]\] and its two equivolume partitions $\bOmega_|=(\Omega_{1,|},\Omega_{2,|})$ into two closed squares 
	and $\bOmega_\backslash=(\Omega_{1,\backslash},\Omega_{2,\backslash})$ into two closed triangles with 
	\[
	\Omega_{1,|}=[a_1,a_1+b]\times [a_2,a_2+b], \qquad 	\Omega_{1,\backslash}=\conv\{
	\ba,\ba+(2b,0),\ba+(0,b)\},
	\]
	 where $\conv$ denotes the convex hull.
	Then 
	\[
	\E \cL_2^2(\cP_{\bOmega_\backslash})<	\E \cL_2^2(\cP_{\bOmega_|}).
	\]	
\end{lem}

\begin{proof}
	%\fp{With the notation of our first paper (which does not reqire that we work with the unit cube) and the paper's Proposition 2}, 
	By Proposition \ref{prop:3}, we have for any equivolume partition $\bOmega=(\Omega_{1},\Omega_{2})$ of $I$
	that 
	\begin{equation}\label{eqInt}
	\E \cL_2^2(\cP_{\bOmega})=\frac14 \sum_{i=1}^2 \int_I q_i(\bx)(1-q_i(\bx))\frac{\dd\bx}{|I|},
	\end{equation}
	where $|I|=2b^2$ and 
	\[
	q_i(\bx)=\frac{|\Omega_i\cap [0,\bx]|}{|\Omega_i|}=b^{-2}|\Omega_i\cap [0,\bx]|.
	\]
	The integral in \eqref{eqInt} is independent of $\ba$, so we may assume $\ba=\bo$ from now on. Furthermore, the substitution $\by=b\bx$ also shows independence of $b$, so we may assume $b=1$. 
    We have 		
	\begin{equation} \label{eg:2dcase}
	q_{1,|}(x_1,x_2)=\min\{x_1,1\}x_2,\qquad q_{2,|}(x_1,x_2)=\max\{x_1-1,0\}x_2, 
	\end{equation}
(see Figure \ref{fig:dtwo})  and thus 
%	\begin{align*}
%	A_{1,|}&=\int_I q_{1,|}(\bx)\,\dd \bx
%	=\int_{\Omega_{1,|}} x_1x_2{\,\dd\bx}+\int_{\Omega_{2,|}}x_2\,\dd\bx=\frac1{4}+\frac12=\frac34,
%	\\
%A_{2,|}&=\int_I q_{2,|}(\bx)\,\dd \bx
%=\int_{\Omega_{2,|}}(x_1-1)x_2\,\dd\bx=\frac14,
%\end{align*}
%and 
	\begin{align*}
B_{1,|}&=\int_I q_{1,|}^2(\bx)\,\dd \bx
=\int_{\Omega_{1,|}} x_1^2x_2^2{\,\dd\bx}+\int_{\Omega_{2,|}}x_2^2\,\dd\bx=\frac1{9}+\frac13=\frac{4}{9},
\\
B_{2,|}&=\int_I q_{2,|}^2(\bx)\,\dd \bx
=\int_{\Omega_{2,|}}(x_1-1)^2x_2^2\,\dd\bx=\frac19. 
\end{align*}

 As $| I | = 2$ we get from \eqref{eq_q2} that 
	\begin{align*}%\label{eqIntcube}
		8\E \cL_2^2(\cP_{\bOmega_|}) &=1-(B_{1,|}+B_{2,|})=\frac49. 
	\end{align*}
	
Furthermore, we have (see Figure \ref{fig:dtwo})
$$
q_{1,\backslash}(\bx)=\begin{cases}
x_1x_2, & \text{ if } \bx \in  \Omega_{1,\backslash}, \\
x_1x_2-\frac{1}{4}(x_1+ 2x_2 - 2)^2, & \text{ if } \bx \in  \Omega_{2,\backslash},
\end{cases}
$$
and
$$
q_{2,\backslash}(\bx)=
\begin{cases}
0, & \text{ if } \bx \in  \Omega_{1,\backslash}, \\
\frac{1}{4}(x_1+ 2x_2 - 2)^2, & \text{ if } \bx \in  \Omega_{2,\backslash}.
\end{cases}
%[(x_2+\tfrac12 x_1-1)_+]^2.
$$
For simplicity we set $\Delta(\bx) := \frac{1}{4}(x_1+ 2x_2 - 2)^2$.
 Using the fact that $\Omega_{1,\backslash}=\{(x_1,x_2)\in I: 0\le x_2\le 1-x_1/2\}$, we obtain  
%	\begin{align*}
%		A_{1,\backslash}&=\int_I q_{1,\backslash}(\bx)\,\dd \bx
%		=\int_{\Omega_{1,\backslash}} x_1x_2{\,\dd\bx}+\int_{\Omega_{2,\backslash}}x_1x_2-\Delta(\bx)\,\dd\bx
%		=\frac1{6}+\frac23=\frac{5}{6},
%		\\
%		A_{2,\backslash}&=\int_I q_{2,\backslash}(\bx)\,\dd \bx
%		=\int_{\Omega_{2,\backslash}}\Delta(\bx)\,\dd\bx=\frac1{6},
%	\end{align*}
%	and 
	\begin{align*}
		B_{1,\backslash}&=\int_I q_{1,\backslash}^2(\bx)\,\dd \bx
		= \int_{\Omega_{1,\backslash}}x_1^2x_2^2{\,\dd\bx}
		 +\int_{\Omega_{2,\backslash}}\big(x_1x_2-\Delta(\bx)\big)^2\,\dd\bx=\frac2{45}+\frac{22}{45}=\frac{24}{45},
		\\
		B_{2,\backslash}&=\int_I q_{2,\backslash}^2(\bx)\,\dd \bx
		=\int_{\Omega_{2,\backslash}}\Delta^2(\bx)\,\dd\bx=\frac1{15}. 
	\end{align*}

Therefore,
\begin{equation*}
	8\E \cL_2^2(\cP_{\bOmega_\backslash})= 
 1-(B_{1,\backslash}+B_{2,\backslash})=\frac{18}{45}=\frac{36}{90}<\frac{4}{9}.   \qedhere
\end{equation*}
\end{proof}

\begin{center}
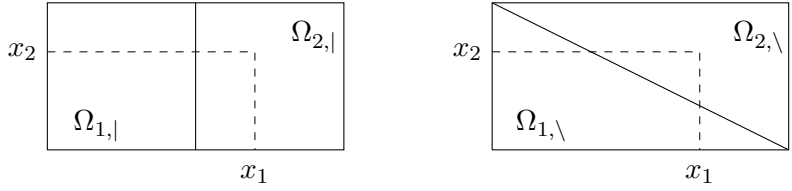
\begin{figure}[h!]
\centering
\begin{tikzpicture}[scale=0.65]
\draw[black, thin] (0,0) -- (6,0) -- (6,3) -- (0,3) -- (0,0);
\draw[black, thin] (3,0) -- (3,3);
\draw[black, thin] (9,0) -- (15,0) -- (15,3) -- (9,3) -- (9,0);
\draw[black, thin] (9,3) -- (15,0);

\draw[black, very thin, dashed] (0,2)--(4.2,2) -- (4.2,0);
\draw[black, very thin, dashed] (9,2)--(13.2,2) -- (13.2,0);

\node at (1, 0.5) {$\Omega_{1,|}$}; 
\node at (5.4, 2.3) {$\Omega_{2,|}$}; 
\node at (10, 0.5) {$\Omega_{1,\backslash}$}; 
\node at (14.4, 2.3) {$\Omega_{2,\backslash}$}; 

\node at (4.2,-0.5) {$x_1$}; 
\node at (-0.5,2) {$x_2$}; 
\node at (13.2, -0.5) {$x_1$}; 
\node at (8.5, 2) {$x_2$}; 
\end{tikzpicture}
\caption{Illustration of the two different partitions of the rectangle $I=[0,2] \times [0,1]$.} \label{fig:dtwo}
\end{figure}
\end{center}

This idea can be extended in a straightforward manner to dimensions $d\geq 2$ as the following lemma shows.

\begin{lem}\label{lem2}
Let $\ba\in \R^d$ and $b>0$ be given. Consider the rectangle 
\begin{equation}\label{eqI}
	I=[a_1,a_1+2b]\times \prod_{i=2}^d [a_i,a_i+b]
\end{equation}
	 and its two equivolume partitions $\bOmega_|=(\Omega_{1,|},\Omega_{2,|})$ into two closed hypercubes 
	and $\bOmega_\backslash=(\Omega_{1,\backslash},\Omega_{2,\backslash})$ into two closed, regular triangular hyperprisms with 
\begin{align*}	
	\Omega_{1,|}&=\prod_{i=1}^d [a_i,a_i+b], \\	
	%\Omega_{1,\backslash} &=\conv\{ \ba,\ba+(2b,0, \ldots, 0),\ba+(0,b, 0, \ldots, 0)\} \times \prod_{i=3}^d [a_i,a_i+b].
	%\Omega_{1,\backslash} &=[a_1, a_1+2b] \times [a_2, -\frac{a_1}{2}+b] \times \prod_{i=3}^d [a_i,a_i+b].
	\Omega_{1,\backslash} &=\conv\{ (a_1, a_2), (a_1, a_1+2b), (a_2, a_2+b) \} \times \prod_{i=3}^d [a_i,a_i+b].
\end{align*}
		Then 
	\revised{
		\begin{equation}\label{eq:difftrinagleTrick}
		\E \cL_2^2(\cP_{\bOmega_|}) - \E \cL_2^2(\cP_{\bOmega_\backslash}) = \frac{1}{20} 3^{-d}> 0,
		\end{equation}
		for all $d\ge 2$. 
	}
\end{lem}

\begin{proof}
%\textcolor{red}{
%\begin{enumerate}
%\item define $Q_1$ and $Q_2$ on $I$ for vertical and diagonal case.
%\item define $\Delta_d$.
%\item calculate the integrals using the above integral identity.
%\item compare the results.
%\end{enumerate}
%}
We start with the vertical case.   For $i=1,2$ we have 
\[
q_{i,|}(\bx)= q_{i,|}(x_1,x_2)\prod_{j=3}^d x_j,
\]
where $ q_{i,|}(x_1,x_2)$ \revised{denotes the two-dimensional function defined in \eqref{eg:2dcase}}. Hence, 
\[
\int_I q_{i,|}^2(\bx)\,\dd \bx =B_{i,|}\int_{[0,1]^{d-2}} \prod_{j=3}^d x_j^2\, \dd(x_3,\ldots,x_d)=\frac1{3^{d-2}}B_{i,|}
\]
where $B_{i,|}$ was calculated in the proof of Lemma \ref{lem1}. As
\[
\int_I |[0,\bx]| \,\dd \bx =\int_{[0,1]^{d-2}} \prod_{j=3}^d x_j\, \dd(x_3,\ldots,x_d)=\frac1{2^{d-2}},
\] 
Equation \eqref{eq_q2} shows 
\[
 8 \ \E \cL_2^2(\cP_{\bOmega_|})=\frac1{2^{d-2}}-\frac1{3^{d-2}}(B_{1,|}+B_{2,|})
 = \frac{4}{2^d} - \frac{5}{3^d}. 
\]
Exactly the same line of arguments applies to the partition with hyperprisms, so 
\[
8 \ \E \cL_2^2(\cP_{\bOmega_\backslash})=\frac1{2^{d-2}}-\frac1{3^{d-2}}(B_{1,\backslash}+B_{2,\backslash})
= \frac{4}{2^d} - \frac{3}{5}\frac{1}{3^{d-2}}. 
\]

We see that
$$ 8 \ \E \cL_2^2(\cP_{\bOmega_|}) - 8 \E \cL_2^2(\cP_{\bOmega_\backslash}) = \frac{2}{5 \cdot 3^d}$$
for all \revised{$d\geq 2$ showing the assertion.}
%\textcolor{red}{After some calculations we get for the diagonal case
%$$\frac{4}{2^d} - \frac{27}{5 \cdot 3^d}$$
%and hence the difference is
%$$ \frac{2}{5 \cdot 3^d}!!$$}
\end{proof}

This lemma combined with  Proposition \ref{prop:3} is the key ingredient in the proof of our main result.

	\begin{proof}[Proof of Theorem \ref{thm:main}] %We start with the case $d=2$.
		Let $\bOmega$ be the partition of  $[0,1]^d$ into $m^d$ congruent closed cubes, where we 
	may assume that the  cube containing the vector $(1,\ldots,1)$ is $\Omega_2$ and the neighboring cube containing $(1-m^{-1},1,1,\ldots,1)$ is $\Omega_1$. 
	%\fpp{[[i think this is not how we name them in the lemmas]]}. 
	
	We compare $\bOmega$ with the partition $\tilde \bOmega$, where 
	$\tilde \Omega_i=\Omega_i$ for $i=3,\ldots, m^d$. The remaining two partition sets 
	must cover  the interval $I$ in \eqref{eqI} with $b=1/m$, $a_1=1-2b, a_2=\cdots=a_d=1-b$,
	and we put $\tilde \Omega_1=\Omega_{1,\backslash}$, $\tilde \Omega_2=\Omega_{2,\backslash}$ with the notation of  Lemma \ref{lem2}. Both partitions are equivolume and consist of convex sets. 
	
	From Proposition \ref{prop:3}, we have 
		\begin{equation}\label{eqIntP3}
		 m^{2d}\E \cL_2^2(\cP_{\bOmega})=\sum_{i=1}^{ m^d} \int_{[0,1]^{d}} q_i(\bx)(1-q_i(\bx)){\dd\bx},
	\end{equation}
  and a corresponding formula for $\tilde \bOmega$, where $q_i(\bx)$ is replaced by $\tilde q_i(\bx)$. For all $\bx\not\in I$, the two integrands coincide. Furthermore, when $\bx\in I$ only the contributions of the first two partitioning sets can differ, so  
  		\begin{align}%\label{eqInt}
  			\nonumber
  	&m^{ 2d}[\E \cL_2^2(\cP_{\bOmega})-\E \cL_2^2(\cP_{\tilde \bOmega})]\\&\qquad =\sum_{i=1}^{2} \int_{I} q_i(\bx)(1-q_i(\bx)){\dd\bx}-\sum_{i=1}^{2}\int_{I} \tilde q_i(\bx)(1-\tilde q_i(\bx)){\dd\bx}
  	\nonumber
  	\\&\qquad=4|I|\left[\E \cL_2^2(\cP_{\bOmega_|})-\E \cL_2^2(\cP_{\bOmega_\backslash})\right]>0,
  	\label{dims}
  \end{align}
where  Proposition \ref{prop:3} and Lemma \ref{lem2}  were used in the last line. 
	\end{proof}

\begin{remark} \label{rem:gain}
\revised{Since $|I| = \frac{2}{m^{d}}$, relations \eqref{eq:difftrinagleTrick} and \eqref{dims} imply  $\E \cL_2^2(\cP_{\bOmega})-\E \cL_2^2(\cP_{\tilde \bOmega}) \in \mathcal{O}(m^{-3d})$, which corresponds to a gain of order $\mathcal{O}(N^{-3})$.}
\end{remark}	

\revised{
\section{Exploiting the local improvement}\label{sec:exploit}
In this section, we discuss the potential and limits of the above idea to replace two neighboring partition squares in a jittered sample by a double-triangular partition, restricting considerations to the  two-dimensional case ($d=2$).   We will show that this modification only improves the mean $\cL_2$-discrepancy if the squares involved are sufficiently close to the upper boundary of $[0,1]^2$. The main goal of the present section is to show that the local ameliorations -- if applied at multiple locations -- improve jittered sampling with $N$ points by a term of order {\bf $N^{-3/2}$} in contrast to the term of order $N^{-3}$ as discussed in Remark \ref{rem:gain}. 
}

\revised{
Although explicit calculations are possible (and will be stated later omitting details), we prefer to give qualitative arguments that reveal the underlying structure of the problem more clearly. 
The key ingredient of this discussion  is concerned with projections of random point sets onto the two axis directions. In the resulting one-dimensional setting, the following two observations describe the worst and the best random point pair distributions explicitly. 
The condition \eqref{equi'} corresponds to the `equivolume condition' when the variables 
$Y_1$ and $Y_2$ stem from a stratification of $[0,1]$ into two sets. 
\begin{lem}\label{lem7}
	Let $Z_1$ and $Z_2$ be i.i.d.~uniform random variables in $[0,1]$. Furthermore,  let $Y_1$ and $Y_2$ be independent random variables in $[0,1]$. If the cumulative distribution functions obey 
		\begin{equation}\label{equi'}
		\frac{F_{Y_1}(x)+F_{Y_2}(x)}2=x,\quad x\in [0,1],
	\end{equation}  then 
\[
\E\cL_2^2(\{Y_1,Y_2\})\le \E\cL_2^2(\{Z_1,Z_2\}),
\]
with equality if and only if $(Y_1,Y_2)$ has the same distribution as $(Z_1,Z_2)$. 	
\end{lem}
\begin{proof}
	If $(Y_1,Y_2)$ are random points derived from an equivolume partition of $[0,1]$ into two sets, this statement is a consequence of the strong partition principle \cite[Theorem 1]{mf21}. In the more general situation considered here, the proof is literally the same if $q_i(x)$ in that proof is replaced by $F_{Y_i}(x)$.  
\end{proof} 
\begin{lem}\label{lem6}
	Let $X_1$ and $X_2$ be independent random variables, uniform in $[0,1/2]$ and $[1/2,1]$, respectively. Furthermore,  let $Y_1$ and $Y_2$ be independent random variables in $[0,1]$. If the cumulative distribution functions obey \eqref{equi'},
then 
	\[
	\E\cL_2^2(\{Y_1,Y_2\})\ge \E\cL_2^2(\{X_1,X_2\}),
	\]
	with equality if and only if $(Y_1,Y_2)$ has the same distribution as $(X_1,X_2)$ or $(X_2,X_1)$. 
\end{lem}
\begin{proof}
	If $(Y_1,Y_2)$ are random points derived from an equivolume partition of $[0,1]$ into two sets, this statement is coinciding with \cite[Corollary 2]{mf21}. In the more general situation considered here, the proof is literally the same if $q_i(x)$ in that proof is replaced by $F_{Y_i}(x)$.  
\end{proof}
}

\revised{These observations are now used to determine the improvement obtained by inserting a double-triangular partition that is not necessarily positioned at the upper right. Let $\bOmega$ be the jittered partition of the unit square with $N=m^2$ sets, $m\ge 2$. 
	We fix a vector $\bz=(z_1,z_2)\in [0,1]^2$ with $z_1\in \{0,\ldots,\frac{m-2}{m}\}$, $z_2\in \{0,\ldots,\frac{m-1}{m}\}$, and compare $\bOmega$ with the partition $\tilde \bOmega$, which only deviates 
	from $\bOmega$ in that two neighboring partitioning cubes at `position' $\bz$ of $\bOmega$ are replaced by triangles. Choosing an appropriate enumeration,  we may put 
	$\tilde \Omega_i=\Omega_i$ for $i=3,\ldots, m^d$. The remaining two partition sets 
	must cover  the interval $I$ in \eqref{eqI} with $b=1/m$, $a_1=1-z_1-2b, a_2=1-z_2-b$,
	and we put $\tilde \Omega_1=\Omega_{1,\backslash}$, $\tilde \Omega_2=\Omega_{2,\backslash}$ with the notation of  Lemma \ref{lem2}. 
	This is illustrated in Figure \ref{fig:markus} (Left). 
	Both partitions are equivolume and consist of convex sets. For $z_1=z_2=0$ the triangular partition is placed as in Figure \ref{fig:def} (Middle).
Equation \eqref{eqIntP3}  
in the proof of Theorem \ref{thm:main}
for $\bOmega$ and the corresponding relation for $\tilde \bOmega$ are 
is still valid with $d=2$. }

\begin{center}
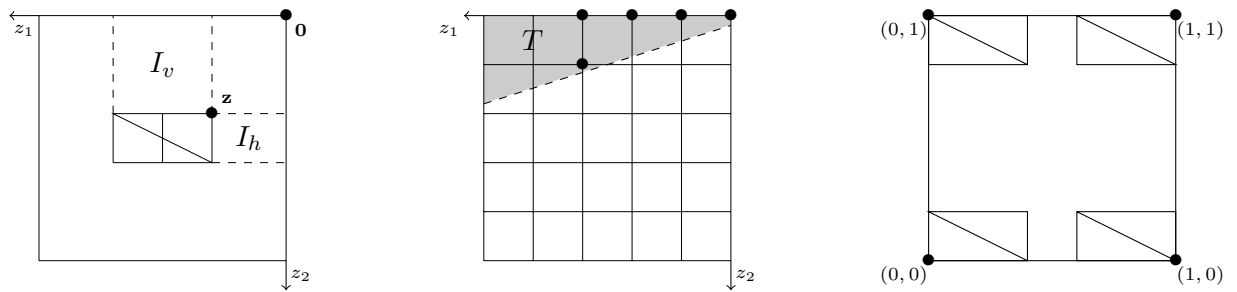
\begin{figure}[h!]
\centering
\begin{tikzpicture}[scale=0.65]

%left image
\draw[very thin] (0,0) -- (5,0) -- (5,5) -- (0,5)--(0,0);
\draw[black,very thin] (1.5,2)--(3.5,2)--(3.5,3)--(1.5,3)--(3.5,2);
\draw[black,very thin] (1.5,2)--(1.5,3);
\draw[black,very thin] (2.5,2)--(2.5,3);

\draw[black,very thin, dashed] (1.5,3)--(1.5,5);
\draw[black,very thin, dashed] (3.5,3)--(3.5,5);

\draw[black,very thin, dashed] (3.5,3)--(5,3);
\draw[black,very thin, dashed] (3.5,2)--(5,2);

\node at (5,5) {$\bullet$}; 
\node at (3.5,3) {$\bullet$}; 

\node at (3.8,3.3) {\tiny $\mathbf{z}$}; 
\node at (5.3, 4.7) {\tiny $\mathbf{0}$}; 
\node at (2.5,4) {$I_v$}; 
\node at (4.25, 2.5) {$I_h$}; 
\node at (-0.3,4.7) {\tiny $z_1$}; 
\node at (5.3,-0.3) {\tiny $z_2$}; 

\draw[thin, <-] (-0.6,5)--(0,5);
\draw[thin, <-] (5, -0.6)--(5,0);

%middle image (no symmetry)
%start by drawing T
\draw[color=white,fill=white!80!black] (9,5) -- (14,5) -- (14,4.8) -- (9,3.2);
\node at (9.99,4.5) {$T$}; 

\draw[step=1cm, black,very thin] (9,0) grid (14,5);

\draw[thin, dashed] (9,3.2)--(14,4.8);

\node at (14,5) {$\bullet$}; 
\node at (13,5) {$\bullet$};
\node at (12,5) {$\bullet$}; 
\node at (11,5) {$\bullet$};
%\node at (12,4) {$\bullet$}; 
\node at (11,4) {$\bullet$}; 
  
\node at (8.3,4.7) {\tiny $z_1$}; 
\node at (14.3,-0.3) {\tiny $z_2$}; 

\draw[thin, <-] (8.6,5)--(9,5);
\draw[thin, <-] (14, -0.6)--(14,0);

% right image
\draw[thin] (18,0) -- (23,0) -- (23,5) -- (18,5)--(18,0);
\draw[thin] (18,0)--(20,0)--(20,1)--(18,1)--(20,0);
\draw[thin] (23,0)--(21,0)--(21,1)--(23,1)--(23,0)--(21,1);
\draw[thin] (23,5)--(21,5)--(21,4)--(23,4)--(21,5);
\draw[thin] (18,5)--(20,5)--(20,4)--(18,4)--(18,5)--(20,4);

\node at (18,0) {$\bullet$}; 
\node at (23,0) {$\bullet$};
\node at (18,5) {$\bullet$}; 
\node at (23,5) {$\bullet$};

\node at (17.5,-0.3) {\tiny $(0,0)$}; 
\node at (23.5,-0.3) {\tiny $(1,0)$};
\node at (17.5,4.7) {\tiny $(0,1)$}; 
\node at (23.5,4.7) {\tiny $(1,1)$};
 
\end{tikzpicture}
\revised{
\caption{\emph{Left:} Coordinates of a rectangle where the jittered partition is modified together with the sets $I_h$ and $I_v$. \emph{Middle:} Illustration of \eqref{eq:good}. The dots represent points $\bz\in T$ at which a rectangle can be placed in order to improve the expected discrepancy. \emph{Right:} Illustration of the different point sets in Table \ref{table1}. Each partition contains exactly one of the four modified rectangles and is otherwise identical with the original jittered partition.}
\label{fig:markus}}
\end{figure}
\end{center}

\revised{
For all $i\ne 1,2$ the corresponding summands associated to $\bOmega$ and $\tilde \bOmega$ coincide. For $i\in \{1,2\}$, they coincide for $\bx\not\in I\cup I_h\cup I_v$ with the `horizontal' and `vertical' rectangular sets
	\[
	I_h=[1-z_1,1]\times [1-z_2-\tfrac1m,1-z_2],\quad
	I_v=[1-z_1-\tfrac2m,1-z_1]\times [1-z_2,1],
	\]
	as in this case, $[0,\bx]$ either is disjoint with $I$ or contains it completely. 
	Hence, 
\begin{equation}\label{eq:deltaDiscr}
	m^4[\E \cL_2^2(\cP_{\bOmega})-\E \cL_2^2(\cP_{\tilde \bOmega})] =\sum_{J\in \{I,I_h,I_v\}}S(J),
\end{equation}
	with 
	\[
	S(J)=\sum_{i=1}^{2} \int_{\revised{J}} q_i(\bx)(1-q_i(\bx)){\dd\bx}-\sum_{i=1}^{2}\int_{\revised{J}} \tilde q_i(\bx)(1-\tilde q_i(\bx)){\dd\bx}.
	\]
	This also explains why we  worked  with a modification of jittered sampling \emph{in the upper right corner} in Section \ref{sec:L2}, as then only the positive term $S(I)$ contributes. 
	As before, 	Proposition \ref{prop:3} and Lemma \ref{lem2}  imply
	\begin{align*}%\label{eqInt}
		S(I)=4|I|\left[\E \cL_2^2(\cP_{\bOmega_|})-\E \cL_2^2(\cP_{\bOmega_\backslash})\right]=c_0 m^{-2},
	\end{align*}
  with $c_0=\frac2{45}$. 
	Putting $J=I_h$ and observing that $q_i(\bx)$ and $\tilde q_i(\bx)$ do not depend on the first entry of $\bx\in I_h$, we obtain
\begin{align}\nonumber
	S(I_h)=&z_1\left[
	\sum_{i=1}^{2} \int_{1-z_2-\tfrac1m}^{1-z_2}
	 q_i(1-z_1,x)(1-q_i(1-z_1,x)){\dd x}
	\right. 
	\\&\quad\left.-\sum_{i=1}^{2}\int_{1-z_2-\tfrac1m}^{1-z_2} \tilde q_i(1-z_1,x)(1-\tilde q_i(1-z_1,x)){\dd x}\right].\label{eq:SIh}
\end{align}
Using 
\[
q_1(1-z_1,x)=q_2(1-z_1,x)=m\big(x-(1-z_2-\tfrac1m)\big)
\]
and a substitution, the first summand in the parenthesis coincides with 
\[
\frac2m \int_0^1 x(1-x)dx=\frac4m \int_0^1 \Var\Big(\frac12 \sum_{i=1}^2 \1_{Z_i\le x}\Big)dx=4\E\cL^2_2(\{Z_1,Z_2\})m^{-1},
\]
where $Z_1,Z_2$ are i.i.d.~uniform in $[0,1]$. This reflects the fact that the projections of the two jitter-points in $I$ onto the $y$-axis are two independent uniform points. A similar argument shows that the second summand on the right side of \eqref{eq:SIh} coincides with 
$4\E\cL^2_2(\{Y_1,Y_2\})m^{-1}$ with two independent non-uniform variables $Y_1$ and $Y_2$ satisfying \eqref{equi'}, so \eqref{eq:SIh}  becomes 
\[
S(I_h)=4z_1[\E\cL^2_2\big[\{Z_1,Z_2\}-\E\cL^2_2(\{Y_1,Y_2\})\big]m^{-1}=c_1z_1m^{-1}
\]
with $c_1>0$ due to Lemma \ref{lem7}. The third contribution $S(I_v)$ in \eqref{eq:deltaDiscr} can be treated similarly, but now the projection of the 
two jitter-points in $I$ onto the $x$-axis yields the best one-dimensional stratification and Lemma \ref{lem6} implies 
\[
S(I_v)=-c_2z_2m^{-1}
\]
with some $c_2>0$. Summarizing, \eqref{eq:deltaDiscr} becomes 
\begin{equation}\label{eq:explicitImprove}
	\E \cL_2^2(\cP_{\bOmega})-\E \cL_2^2(\cP_{\tilde \bOmega})=c_0m^{-6}+c_1z_1m^{-5}-c_2z_2m^{-5}, 
\end{equation}
where all occurring constants are positive.
}

\revised{
This result shows that one can in general not improve jittered sampling by replacing two arbitrary horizontally neighboring jitter squares with triangles. More specifically, we get an improvement if and only if the  numbers $z_1,z_2$ describing the upper right corner satisfy 
\begin{equation}\label{eq:good}
z_2<\frac{c_0}{c_2}m^{-1}+\frac{c_1}{c_2}z_1. 
\end{equation}
	Geometrically, \eqref{eq:good} is satisfied if and only if $\bz$ is an element of the set 
	\[
	T=\{\bz\in [0,1]^2: z_2<\tfrac{c_0}{mc_2}+\tfrac{c_1}{c_2}z_1\};
	\]
see Figure \ref{fig:markus} (Middle) for the case $m=5$.
}

\revised{To illustrate our results, it is necessary to explicitly know the involved constants. We obtain
\begin{equation*}%\label{constants}
c_0 = \frac{2}{45}, \ \ \ c_1=\frac{1}{15}, \ \ \ c_2 = \frac{1}{5},
\end{equation*}
where $c_0$ was determined in Lemma \ref{lem2}, and $c_1, c_2$ were derived by elementary calculations. 
	In accordance with the qualitative arguments above, $c_0,c_1,c_2$ and hence the slope of the bounding line in the definition of $T$ are independent of $m$.}

\revised{In Table \ref{table1} we present numerical results to further illustrate \eqref{eq:good}. We compare the empirical mean of the $\cL_2$-discrepancy of 1000 individual point sets for each $N$. In particular, we compare instances of jittered sampling $\cP_{\mathrm{jit}}$ to modified point sets in which we replaced exactly two adjacent boxes with triangles. To be more precise, we moved the original modified rectangle once into each of the 4 corners of the unit square.
We denote these different sets simply by $\cP_{xy}$ indicating which of the four vertices $(x,y)$ of the unit square is a vertex of the modified rectangle.
We see that moving the rectangle to the upper left corner increases the gain; i.e. considering $\bz=(1-2/m,0)$ instead of the original $\bz=(0,0)$.
Furthermore, moving the rectangle to the lower right corner worsens the result; i.e. considering $\bz=(0,1-1/m)$.
Of course, this is both in line with our theoretical analysis.}

\revised{For example, for $P_{01}$ we expect a gain of size $c_1 m^{-5}$ according to \eqref{eq:explicitImprove}. As $c_1=1/15$ and $N=m^2=10^2$ the expected gain is 
%approximately $0.67\cdot 10^{-6}$, 
$\mathcal{O}(10^{-6})$
which is in correspondence with the 
empirical gain of $\approx 0.857\cdot 10^{-6}$.
%\footnote{To be sure: you write "Note that the absolute constant $c_1$ is itself of order $\mathcal{O}(10^{-1})$ and hence we see a difference of order $\mathcal{O}(10^{-6})$; i.e. the difference is $\approx 0.857\cdot 10^{-6}$." Do you mean that the empirical difference is $\approx 0.857\cdot 10^{-6}$?} 
Similarly, for $P_{11}$, we expect a gain of size $c_0 m^{-6}$, i.e. of order $\mathcal{O}(10^{-7})$, which corresponds to the empirical 
difference $\approx 0.85 \cdot 10^{-7}$.}

\begin{table}[h]
\begin{center}
%\begin{tabular}{|c|| l || l | l || l || l |}
\begin{tabular}{|c || c | c | c | c | c |}
\hline
%&random&\multicolumn{2}{l ||}{vertical lines} & equivol lines & jittered\\
%&random&vertical lines& equivol lines & jittered\\
$ \E {\cL_2}( \cdot )$ & $\cP_{\mathrm{jit}}$ & $\cP_{11}$& $\cP_{01}$ & $\cP_{10}$ & $\cP_{00}$\\
\hline
$N=7^2$ & 0.000476834 & 0.00047629&0.000473918  & 0.000486402&0.000481983\\
\hline
$N=10^2$ & 0.00016377 &0.000163685 & 0.000162913 & 0.000165225&0.000165369 \\
\hline
$N=14^2$ & 0.0000599499 &0.0000599455 &0.0000598861 &0.0000601582&0.0000602246 \\
%$N=20^2$ & 0.0000207618 & 0.0000207663& 0.000020718 &0.0000207889 &0.000020794\\
%$30^2 = 900$ & $6.1301\cdot10^{-6}$ &$6.13001\cdot10^{-6}$ &$6.12501\cdot10^{-6}$  &$6.13464\cdot10^{-6}$ & $6.14072\cdot10^{-6}$\\
\hline
\end{tabular}
\medskip 

\end{center}
\label{table1}
\caption{Comparison of expected $\cL_2$-discrepancy of classical jittered sampling with our different modifications. The values in the table give the empirical mean of the $\cL_2$-discrepancy of 1000 individual samples for each $N$. We calculated the discrepancy of individual samples with Warnock's formula.
}

\end{table}

\revised{Now, modifying the jittered partition successively at several positions leads to an accumulated improvement. 
In fact, note that the proof of \eqref{eq:explicitImprove} still works even if an arbitrary partition $\bOmega$ instead of the jitter partition is used, as long as $\bOmega$ coincides with the jittered partitions within $I$. Hence, replacing every second rectangle with a position $(z_1,z_2)$ obeying \eqref{eq:good} in a jittered partition  with  a double-triangular partition, the overall gain is 
\[
g=
\sum_{z_1\in \frac1m \big\{0,2,\ldots,2\lfloor \frac{m-2}{2}\rfloor\big\}}
\sum_{z_2\in \frac1m\{0,\ldots,m-1\}}\1_{(z_1,z_2)\in T}
[c_0m^{-6}+c_1z_1m^{-5}-c_2z_2m^{-5}].
\]
As $0\le z_1,z_2\le 1$, we have  $[c_0m^{-6}+c_1z_1m^{-5}-c_2z_2m^{-5}]\le c_3 m^{-5}$ for all $m$ (where we have put $c_3=c_0+c_1$), so 
\[
g\le c_3 m^{-5} \big(2\big\lfloor \frac{m-2}{2}\big\rfloor+1\big)m\le c_3m^{-3}. 
\]
To show that this is the correct rate of convergence note that the rectangle 
\[
 R=\{(z_1,z_2)\in [0,1]^2: z_1\ge \tfrac12, z_2\le \tfrac{c_1}{4c_2}\}
\]
is contained in $T$ for all $m$, and that $\bz\in R$ implies  $[c_0m^{-6}+c_1z_1m^{-5}-c_2z_2m^{-5}]\ge \tfrac{c_1}{4} m^{-5}$, so 
\[
g\ge \tfrac{c_1}{4}m^{-5}\sum_{z_1\in \frac1m \big\{0,2,\ldots,2\lfloor \frac{m-2}{2}\rfloor\big\}}
\sum_{z_2\in \frac1m\{0,\ldots,m-1\}}\1_{(z_1,z_2)\in R}\ge c_4 m^{-3}
\]
for some constant $c_4>0$. Putting things together, we see that the gain $g$ behaves like $m^{-3}=N^{-3/2}$ as $m\to \infty$. Note that symmetry considerations with respect to the main diagonal would allow to modify almost double as many rectangles,	
but this will of course not change the asymptotic order of the gain. Moreover, for small $m$, i.e. $m<150$, the gain is not as big due to effects of the absolute constants. Concluding, exploiting the local modifications suggested in this paper to their limit
yields an improvement of jittered sampling with $N$ points in the order of $N^{-3/2}$.  
}

\revised{Table \ref{table:gain} shows the result of a second numerical experiment for $N=10^2$. We generated 10000 instances of jittered samples, and compare them to different modified point sets based on a particular jittered sampling set; i.e. we compare $\cP_{\mathrm{jit}}$ to $\cP_{01}$ as defined above as well as to $\cP_{z_2=0}$ which is a partition in which we modify all pairs of rectangles in the top row and to $\cP_{\mathrm{all}}$ in which all eligible rectangles are replaced by triangles.
Note that in the present case \eqref{eq:good} is satisfied for every $\bz$ in
$$\left \{\bz \in [0,1]^2: z_2 < \frac{1}{45} + \frac{1}{3} z_1 \right \}.$$
}

\begin{table}[h]
\begin{center}
%\begin{tabular}{|c|| l || l | l || l || l |}
\begin{tabular}{|c || c | c |c| c|}
\hline
%&random&\multicolumn{2}{l ||}{vertical lines} & equivol lines & jittered\\
%&random&vertical lines& equivol lines & jittered\\
$ \E {\cL_2}( \cdot )$ & $\cP_{\mathrm{jit}}$ & $\cP_{01}$& $\cP_{z_2=0}$ &$\cP_{\mathrm{all}}$\\
\hline
$N=10^2$ &0.00016366 &0.000163152 &0.000162172& 0.00016101 \\
\hline

\end{tabular}
\medskip 

\end{center}
\caption{\label{table:gain}
	Comparison of expected $\cL_2$-discrepancy of classical jittered sampling with point sets stemming from different modifications. The values in the table give the empirical mean of the $\cL_2$-discrepancy of 10000 individual samples. We calculated the discrepancy of individual samples with Warnock's formula.
}

\end{table}

%%%%%%%%%%%%%%
\section{Concluding remarks and open problems}
\label{sec:conclusion}

In this final section we collect various open problems for future research.

\begin{enumerate}
\item {\bf $\cL_p$-discrepancy.} It is of course natural to ask whether our result also holds for the expected $\cL_p$-discrepancy.  The advantage of 
the case $p=2$ is that the contributions of the individual sampling points 
to the mean discrepancy behave additively due to Proposition \ref{prop:3}. This is not the case for $p\ne 2$, and already the case $p=4$, for which an analogue \cite[Proposition 2]{mf21} of Proposition \ref{prop:3} is known,  does not have this simple structure. The proof of an extension of Theorem \ref{thm:main} to $p\ne 2$ appears therefore to require a substantially new ingredient. We recall that the main tool utilised in the proof of the Strong Partition Principle \cite[Theorem 1]{mf21}, is an inequality due to Hoeffding. However, it appears that generalisations of Hoeffding's result based on the theory of majorisations \cite{MOA} do not suffice to extend our Theorem \ref{thm:main}.
\item {\bf Star discrepancy. } Naturally, we are not only interested in the expected $\cL_p$-discrepancy, but also in a related result for the star discrepancy. However, at the moment even proving a Strong Partition Principle for the star discrepancy seems out of reach.
\item {\bf Generalized $\cL_p$-discrepancy.}
\revised{We recall that the main idea of Hickernell's generalization is to not only consider the ordinary $\cL_p$-discrepancy of a point set, but also the discrepancies of all projections to lower dimensional faces of the unit cube. In the two-dimensional case this means that we also need to include the discrepancy of the projections of the point set to the $x$- and the $y$-axis. Direct calculations seem to indicate that our construction does not improve jittered sampling with respect to Hickernell's notion. %Qualitatively, the answer can be investigated in the light of Lemma \ref{lem7} and Lemma \ref{lem6}. 
We leave it as an open question whether there is another construction that can actually improve the discrepancy of jittered sampling with respect to Hickernell's notion.}

%\fp{Two problems. L2 can be erratic in a certain way (mat), i.e. points that cluster at upper right are nearly optimal and we can not prove a koksma/hlawka thm. hickernell remedies this situation by defining the generalized Lp. (hick2). the idea is to not only consider Lp, but also Lp of projections of the point set to lower dimensional faces of the unit cube. in the special case $d=2$, 
%we note that we can use again lemma 7 and lemma 8 for the analysis. it seems that our construction does not improve jittered sampling in the Hickernell discrepancy. We leave it as an open question whether there is another construction that actually can improve it.}
\item {\bf Asymptotic gain.} Finally, our result is of theoretical interest as it shows the existence of stratified samples which improve classical jittered sampling. \revised{But our gain concerns only lower order terms of the expected discrepancy.}
%\footnote{Have we stated the right order of jittered sampling somewhere? Maybe we should, so we can put our gain in relation to it.} 
It would be very interesting to know whether jittered sampling has the optimal order of magnitude or whether there are stratified point sets with $N=m^d$ points with an asymptotic gain over classical jittered sampling.
%\fp{\item {\bf Asymptotic gain revised.} Finally, our result is of theoretical interest as it shows the existence of stratified samples which improve classical jittered sampling. The following crude analysis shows that our maximal gain has the best possible order in $N$ up to logarithmic factors. It is well known that 
%$$\cL_2(\cP) \leq \cL_{\infty}(\cP),$$
%see for example \cite[Remark 3.21]{DP}. 
%By the results of Doerr \cite{doerr} we have $\E \cL_{\infty} (\cP_N) = \mathcal{O}(N^{-3/4+\varepsilon})$ in the two dimensional setting and for $\varepsilon > 0$. This implies that the square of the expected $%\cL_2$-discrepancy satisfies 
%$$\E \cL_{2}^2 (\cP_N) = \mathcal{O}(N^{-3/2+\varepsilon}).$$
%It would be very interesting to know whether jittered sampling has the optimal order of magnitude or whether there are stratified point sets with $N=m^d$ points with an asymptotic gain over classical jittered sampling.
%}
\end{enumerate}

%%%%%%%%%%%%%%%%%%%
% References
%%%%%%%%%%%%%%%%%%%

\appendix
%\section{Numerical results}

%In the following we present numerical results \revised{showing the %gain with our constructions at different positions $\bz$.}

%\section{New figures}

%\fp{ 
%	TO DO:}
%\fp{
%	\begin{enumerate}
		%\item \revised{A phrase or two in the introduction should mention the new Section \ref{sec:exploit} and its main message. Possibly also an acknowledgment for the reviewer who asked for exploiting the potential of iterating our construction leading to the new section (?)} 
		%\item \revised{A comment on other discrepancies is missing (Hickernell or others) on p.~2.}
		%\item \revised{A comment after \eqref{eq:good}, where it is explained how the numerical results in Table 1 illustrate the theoretical findings. In that table the notation $\cP_{ij}$ is used but not explained and one might explain it in the caption. I also wonder if this table should be in an appendix or just included in Section \ref{sec:exploit}.} 
		%\item Discuss star discrepancy case a bit more -- if helpful, add numerical results in a second table. If we do not have much to say, this can remain in the conclusion section.
%\item \revised{I suggest that all changes in the revised version for the referees will be colored (only one color -- it does not matter for them who has written the changes). Hence, Table 1 should be colored.}
		%\item \revised{Should the comment on 'asymptotic gain' in the conclusion be adjusted? }
%	\end{enumerate}
%}

\end{document}